\documentclass{amsart}
\usepackage{graphicx}
\usepackage{amsmath}
\newtheorem{theorem}{Theorem}[section]
\newtheorem{corollary}{Corollary}[section]
\newtheorem{proposition}{Proposition}[section]
\newtheorem{lemma}{Lemma}[section]

\newtheorem{remark}{Remark}[section]
\theoremstyle{definition}
\newtheorem{definition}{Definition}[section]
\newtheorem{example}{Example}
\numberwithin{equation}{section}

\newcommand{\opD}{\operatorname{D_{pq}}}
\begin{document}

\title[On a New Type $pq$-Calculus]{On a New Type $pq$-Calculus}

\author{\.{I}lker GEN\c{C}T\"{U}RK}

\address{K{\i }r{\i }kkale University, Department of Mathematics, 71450, K{\i }r{\i }kkale, Turkey}%
\email{ilkergencturk@gmail.com}


\begin{abstract}
In this paper, we introduce a new type of $ pq $-calculus. The $ pq $-derivative and $ pq $-integration are investigated and various properties of these concepts are given. The fundamental theorem of $ pq $-calculus and formulas of $ pq $-integration by part are also presented.
\end{abstract}
\maketitle
\section{Introduction}
Classically, the derivative of a function $f(x)$ is defined by
\begin{equation}
f'(x)=\lim\limits_{y \to x} \frac{f(y)-f(x)}{y-x}.
\end{equation}

In 2002, Kac et al. published a work in which they described two types of quantum calculus, the $ q $-calculus and $ h $-calculus. The definitions of $ q $-derivative and $ h $-derivative of $ f(x) $ come with, if we do not take limit in derivative of a function $ f $, and taking $y=qx $ and $ y=x+h $, respectively. In \cite{KacCheung}, one can see that more details about these two types of quantum calculus. Also for further information about a history of $ q $-calculus, we refer to readers to \cite{Ernsthist}.

In recent years, researchers have shown an increasing interest in  $ q $-calculus.  A considerable amount of literature has been published on this subject. These studies have examined interesting applications in various mathematical areas such as fractional calculus, discrete function theory, umbral calculus, etc. See \cite{AnnMan,Ernstcomp,Ernstq,KocaGenAyd,MumNaj}  Also, as a special form of $ q $-calculus, $ pq $-calculus have been much attractive. See \cite{KocaGen,Sad}.

As a similar manner, if we avoid taking limit and also take $ y=x^p $, where $ p $ is different from $ 1 $, i.e., by considering the following expression
\begin{equation}\label{p-der}
	\frac{f(x^p)-f(x)}{x^p-x},
\end{equation}
we come across with a new type of quantum calculus, the $ p $-calculus, and also \eqref{p-der}, as definition of the $ p $-derivative, is given. The formula \eqref{p-der} and some properties of it are defined and investigated in \cite{NeaMeh}. In \cite{NeaMeh2}, authors bring forward some new properties of functions in $ p $-calculus.

Accordingly, the objective of the present work paper is to develop notation and terminology for a new type post quantum calculus, namely the $ pq $-calculus. By with appropriate choices, all results in this paper can be reduced in \cite{NeaMeh}.

The main finding of this paper can be summarized as follows. In section 2, $ pq $-derivative is described and given its attributes. In section 3, we present the $ pq $-antiderivative, $ pq $-integral and investigated their convergence. In the following section, we express definite the $ pq $-integral and improper $ pq $-integral. In the last section, the fundamental theorem of $ pq $-calculus is proved and formulas of $ pq $-integration by parts is obtained.

\section{$pq$-derivative}

In this section, we introduce $pq$-differential, $pq$-derivative and investigated some of their properties. Throughout this paper, it has been assumed that $ p $ and $ q $ are a fixed number different from 1 and domain of function $ f(x) $ is $ [0,\infty) $.

\begin{definition}
The $pq$-differential of the function $f$ is defined by

\begin{equation}\label{pqdiffdef}
d_{pq}f(x)=f(x^p)-f(x^q).
\end{equation}
In particular, $d_{pq}x=x^p-x^q$. By using \eqref{pqdiffdef}, it can be given definition of $pq$-derivative of a function .

\end{definition}

\begin{definition}
The $pq$-derivative of the function $f$ is defined by
\begin{equation}\label{pqderdef}
	\opD f(x)=\frac{f(x^p)-f(x^q)}{x^p-x^q},\quad x \neq 0,1,
\end{equation}
and
\begin{equation}
\opD f(0)=\lim\limits_{x \to 0^+} \opD f(x),  \opD f(1)=\lim\limits_{x \to 1} \opD f(x).
\end{equation}

\end{definition}

\begin{definition}
The $pq$-derivative of higher order of a function $f$ is defined by for $n \in \mathbb{N}$
\begin{equation}
\left(\opD^0 f\right)(x)=f(x), \left(\opD^n f\right)(x)=\opD~\left(\opD^{n-1}f\right)(x)
\end{equation}
\end{definition}

\begin{example}
	Let $f(x)=c, g(x)=x^n$ and $h(x)=\ln(x)$ where $c$ is a constant and $n \in \mathbb{N}$. Then we obtain
	\begin{itemize}
		\item $\opD~f(x)=0$,
		\item\[\opD~g(x)=\frac{g(x^p)-g(x^q)}{x^p-x^q}=\frac{x^{pn}-x^{qn}}{x^p-x^q}=\frac{x^{(p-1)n}-x^{(q-1)n}}{x^{p-1}-x^{q-1}}x^{n-1},\]
		\item \[\opD~h(x)=\frac{h(x^p)-h(x^q)}{x^p-x^q}=\frac{(p-q)\ln(x)}{x^{p}-x^{q}}=\frac{(p-q)\ln(x)}{x^{p-1}-x^{q-1}}\frac{1}{x}.\]
	\end{itemize}
\end{example}
Also we note that $pq$-derivative of a function is a linear operator. More precisely, $\opD$ has the property that for any constants $a$ and $b$ and arbitrary functions $f(x)$ and $g(x)$,

\begin{equation}
	\opD(af(x)+bg(x))=a \opD f(x)+b \opD g(x)
\end{equation}

\begin{proposition}
$pq$-derivative has the following product rules
\begin{eqnarray}
\opD\left(f(x) g(x)\right)&=& g(x^p)\opD f(x)+f(x^q)\opD g(x) \label{prrule1} \\
\opD\left(f(x) g(x)\right)&=& g(x^q)\opD f(x)+f(x^p)\opD g(x) \label{prrule2}
\end{eqnarray}
\end{proposition}
\begin{proof}
From definition of the $pq$-derivative, it can be seen that
\begin{eqnarray*}
\opD\left(f(x) g(x)\right)&=& \frac{f(x^p)g(x^p)-f(x^q)g(x^q)}{x^p-x^q} \\
&=& \frac{g(x^p)[f(x^p)-f(x^q)]+f(x^q)[g(x^p)-g(x^q)]}{x^p-x^q} \\
&=&g(x^p)\opD f(x)+f(x^q)\opD g(x).
\end{eqnarray*}
This proves \eqref{prrule1}. Similarly, by symmetry, \eqref{prrule2} can be deduced.
\end{proof}
\begin{proposition}
	$pq$-derivative has the following quotient rules
	\begin{eqnarray}
	\opD\left(\frac{f(x)}{g(x)}\right)&=& \frac{g(x^q)\opD f(x)-f(x^q)\opD g(x)}{g(x^q)g(x^p)} \label{quorule1} \\
	\opD\left(\frac{f(x)}{g(x)}\right)&=& \frac{g(x^p)\opD f(x)-f(x^p)\opD g(x)}{g(x^q)g(x^p)} \label{quorule2}
	\end{eqnarray}
\end{proposition}
\begin{proof}
	The proof of these rules can be obtained by using \eqref{prrule1} and \eqref{prrule2} with functions $\frac{f(x)}{g(x)}$ and $g(x)$.
\end{proof}

\section{$pq$-integral}
In this section, we formalize the notion of $pq$-integral.  From now on, it can be assumed that $ 0<q<p<1 $.

\begin{definition}
The function $F(x)$ is a $pq$-antiderivative of $f(x)$ if $\opD F(x)=f(x)$. It is donated by
\begin{equation}\label{pqint}
	\int f(x) d_{pq}x.
\end{equation}
\end{definition}

Here, we discuss how to construct $pq$-integral which gives us $pq$-antiderivative $F(x)$ of an arbitrary function $f(x)$, i.e. $\opD F(x)=f(x)$. For this, we need two operators which are $\widetilde{M_p}, \widetilde{M_q} $ defined by $\widetilde{M_p}f(x):=f(x^p), \widetilde{M_q}f(x):=f(x^q)$.
From the definition of $pq$-derivative, by using operators $\widetilde{M_p}, \widetilde{M_q}$, we have

\begin{equation}
\opD F(x)=\frac{1}{x^p-x^q}\left(\widetilde{M_p}-\widetilde{M_q}\right)F(x)=f(x),
\end{equation}
then  by using geometric series expansion, we have

\begin{eqnarray}\label{pqintseri}
	F(x)&=&\frac{1}{\left(\widetilde{M_p}-\widetilde{M_q}\right)}((x^p-x^q)f(x)) \notag\\
	&=&\frac{1}{\widetilde{M_p}\left(1- \frac{\widetilde{M_q}}{\widetilde{M_p}}\right)}((x^p-x^q)f(x)) \notag \\
	&=& \sum_{j=0}^{\infty} \left(\widetilde{M_p}^{-j-1} \widetilde{M_q^j}\right) ((x^p-x^q)f(x)) \notag \\
	&=&\sum_{j=0}^{\infty} \left( x^{q^j p^{-j}}- x^{q^{j+1} p^{-j-1}} \right) f\left(x^{q^j p^{-j-1}}\right).
\end{eqnarray}

\begin{definition}
	The $pq$-integral of $f(x)$ is defined with serial expansion
	
	\begin{equation}\label{pqintdef}
		\int f(x) d_{pq}x= \sum_{j=0}^{\infty} \left( x^{q^j p^{-j}}- x^{q^{j+1} p^{-j-1}} \right) f\left(x^{q^j p^{-j-1}}\right).
	\end{equation}
\end{definition}

From this definition, it can be obtained a more general formula:

\begin{eqnarray*}
\int f(x) \opD g(x) d_{pq}x&=& 	 \sum_{j=0}^{\infty} \left( x^{q^j p^{-j}}- x^{q^{j+1} p^{-j-1}} \right) f\left(x^{q^j p^{-j-1}}\right) \opD g\left(x^{q^j p^{-j-1}}\right) \notag \\
&=&  \sum_{j=0}^{\infty} \left( x^{q^j p^{-j}}- x^{q^{j+1} p^{-j-1}} \right) f\left(x^{q^j p^{-j-1}}\right) \frac{g\left(\left(x^{q^j p^{-j-1}}\right)^p\right)-g\left(\left(x^{q^j p^{-j-1}}\right)^q\right)}{\left( x^{q^j p^{-j}}- x^{q^{j+1} p^{-j-1}} \right)} \notag \\
&=& \sum_{j=0}^{\infty} f\left(x^{q^j p^{-j-1}}\right) \left(g\left(x^{q^j p^{-j}}\right)-g\left(x^{q^{j+1} p^{-j-1}}\right)\right)
\end{eqnarray*}
or otherwise stated
\begin{equation}\label{pqintgeneral1}
\int f(x) d_{pq}g(x)= \sum_{j=0}^{\infty} f\left(x^{q^j p^{-j-1}}\right) \left( g\left(x^{q^j p^{-j}}\right)-g\left(x^{q^{j+1} p^{-j-1}}\right)\right).
\end{equation}

It is noted that \eqref{pqintdef} is formal because the series does not always converges. The following theorem gives a sufficient condition for under what conditions it really converges to a $pq$-antiderivative.

\begin{theorem}\label{coninttheo}
 If $|f(x)x^{\alpha}|$ is bounded on the interval $\left(0,A\right]$ for some $0 \leq \alpha <1$, at that point the  $pq$-integral \eqref{pqintdef} converges to a function $H(x)$ on $\left(0,A\right]$, which is a $pq$-antiderivative of $f(x)$. In addition, $H(x)$ is continuous at $x=1$ with $H(1)=0$.
\end{theorem}

\begin{proof}
	For the correctness of the theorem, we deal with following two cases. \newline
	\textbf{Case 1.} Let $x \in  \left(1,A\right]$ be. By assuming that $|f(x)x^{\alpha}|< M $ on $\left(1,A\right]$, it can be obtain for $0 \leq j$,
	\[
	\left|f\left(x^{q^j p^{-j-1}}\right)\right|< M \left(x^{q^j p^{-j-1}}\right)^{-\alpha} < M.
	\]

	Thus for any $1 < x \leq A$, we have
	\[
	\left|  \left(x^{q^j p^{-j}}- x^{q^{j+1} p^{-j-1}}\right) f\left(x^{q^j p^{-j-1}}\right) \right| \leq  \left(x^{q^j p^{-j}}- x^{q^{j+1} p^{-j-1}}\right) M.
	\]
	
	Since $1-\alpha > 0$ and $0<q<p<1$,
	\[
	 \sum_{j=0}^{\infty}  \left( x^{q^j p^{-j}}- x^{q^{j+1} p^{-j-1}} \right) M~=~M(x-1),
	\]
it can be seen that the series,  by a comparison test,  $pq$-integral converges to a function $F(x)$. By \eqref{pqintseri},  it shows directly that $F(1)=0$.

For   $1 < x \leq A$,
\[
\left|F(x) \right|=\left| \sum_{j=0}^{\infty}  \left( x^{q^j p^{-j}}- x^{q^{j+1} p^{-j-1}} \right) f\left(x^{q^j p^{-j-1}}\right) \right|\leq M(x-1)
\]
which tends to $0$ when $x \to 1^+$. Since $F(1)=0$, thereby $F$ is right continuous at $x=1$.

To verify that $F(x)$ is a $pq$-antiderivative of $ f(x) $, we $pq$-differentiate it:

\begin{eqnarray*}
	\opD F(x)&=&\frac{F(x^p)-F(x^q)}{x^p-x^q} \notag \\
	&=& \frac{\sum\limits_{j=0}^{\infty} \left( x^{q^j p^{1-j}}-x^{p^{-j} q^{j+1}}   \right) f\left(x^{q^j p^{-j}}\right)-\sum\limits_{j=0}^{\infty} \left( x^{q^{j+1} p^{-j}}-x^{p^{-j-1} q^{j+2}}   \right) f\left(x^{q^{j+1} p^{-j-1}}\right)}{x^p-x^q} \notag \\
	&=& \frac{\sum\limits_{j=0}^{\infty} \left( x^{q^j p^{1-j}}-x^{p^{-j} q^{j+1}}   \right) f\left(x^{q^j p^{-j}}\right)-\sum\limits_{j=1}^{\infty} \left( x^{q^{j} p^{-j+1}}-x^{p^{-j} q^{j+1}}   \right) f\left(x^{q^{j} p^{-j}}\right)}{x^p-x^q} \notag \\
		&=&f(x). \\
		&&
\end{eqnarray*}

\textbf{Case 2.} Let $x \in (0,1)$ be.  By assuming  $|f(x)x^{\alpha}|< M $ on $(0,1)$, it can be obtained for $0 \leq j$,
\[
\left| f\left(x^{q^j p^{-j-1}}\right) \right| < M \left(x^{q^j p^{-j-1}}\right)^{-\alpha} \leq M x^{-\alpha}.
\]

Hence for any $0<x<1$, we have
\[
\left|  \left(x^{q^j p^{-j}}- x^{q^{j+1} p^{-j-1}}\right) f\left(x^{q^j p^{-j-1}}\right) \right|\leq  \left(x^{q^{j+1} p^{-j-1}}-x^{q^j p^{-j}} \right) M x^{-\alpha}.
\]

Since $1-\alpha > 0$ and $0<q<p<1$,
\[
\sum_{j=0}^{\infty}  \left(x^{q^{j+1} p^{-j-1}}-x^{q^j p^{-j}} \right)  M x^{-\alpha}=M x^{-\alpha} (1-x),
\]
hence, by comparison test $ pq $-integral converges to a function $ G(x) $ and by \eqref{pqintseri} it shows directly   that $G(1)=0$. Similar to proof Case 1, it can be seen that $ G $ is left continuous at $ x=1 $ and is also $ pq $-antiderivative of $ f(x) $.

By bringing together above two cases, it is now possible to define a new function $ H(x) $ by
\[
H(x)=G(x){\chi_{(0,1)}}(x)+F(x){\chi_{(1,A]}}(x)
\]
where $ \chi_{I} $ is characteristic function on an interval $ I $.

As can be easily seen from function $H(x)$, $ pq $-integral converges to $ H(x) $ on $ (0,A] $ and it is also worth noting that $ H(x) $ is $ pq $-antiderivative of $ f(x) $ on $ (0,1) \cup (1,A] $ and is continuous at $ x=1 $ with $ H(1)=0 $.

If $ f(x)  $ continuous at $ x=1 $, then $ \opD H(1)=\lim\limits_{x \to 1} H(x)=f(1)$ and we conclude that $ H(x) $ is $ pq $-antiderivative of $ f(x) $  on $ (0,A] $, which proves the theorem.
\end{proof}

\begin{remark}
It has been demonstrated  that if it satisfies the assumptions of Theorem \ref{coninttheo}, then the $ pq $-integral gives the unique $ pq $-antiderivative which is continuous at $ x=1 $, up to a constant. In contrast, if we know that $ F(x) $ is a $ pq $-antiderivative of $ f(x) $ and $ F(x) $ is continuous at  $ x=1 $, $ F(x) $ must be given, up to a constant, by aid of \eqref{pqintdef}, since a partial sum of the $ pq $-integral is

\begin{equation*}
\sum_{j=0}^{N} \left( x^{q^j p^{-j}}- x^{q^{j+1} p^{-j-1}} \right) f\left(x^{q^j p^{-j-1}}\right)=\sum_{j=0}^{N} \left( x^{q^j p^{-j}}- x^{q^{j+1} p^{-j-1}} \right) \opD F(t)|_{t=x^{q^j p^{-j-1}}}
\end{equation*}

\begin{eqnarray*}
&=& \sum_{j=0}^{N} \left( x^{q^j p^{-j}}- x^{q^{j+1} p^{-j-1}} \right) \left(  \frac{F\left(x^{q^{j}p^{-j}}\right)-F\left( x^{q^{j+1}p^{-j-1}} \right)}{\left(x^{q^{j}p^{-j}}-x^{q^{j+1}p^{-j-1}}\right)}\right) \\
&=&\sum_{j=0}^{N} \left( F\left(x^{q^{j}p^{-j}}\right)-F\left( x^{q^{j+1}p^{-j-1}} \right) \right)  \\
&=&F(x)-F\left(x^{q^{N+1}p^{-N-1}}\right)
\end{eqnarray*}
which tends to $ F(x)-F(1) $ as $ N $ tends to $ \infty $, by the continuity of $ F(1) $ at $ x=1 $.
\end{remark}

In this theorem, boundedness of the function can not be removed. In fact, let $\displaystyle{ f(x)=\frac{1}{x^p-x^q} }$ be. By using definition of $ pq $-integral, one can get
\begin{equation*}
	\sum_{j=0}^{\infty} \left( x^{q^j p^{-j}}- x^{q^{j+1} p^{-j-1}} \right) f\left(x^{q^j p^{-j-1}}\right)=\sum_{j=0}^{\infty} \left( x^{q^j p^{-j}}- x^{q^{j+1} p^{-j-1}} \right)\frac{1}{\left( x^{q^j p^{-j}}- x^{q^{j+1} p^{-j-1}}\right)}=\infty
\end{equation*}
The formula fails because $ f(x)x^{\alpha} $ is not bounded for any $0 \leq \alpha <1  $ on $ (0,1)\cup (1,A] $.

\section{The definite $ pq $-integral}

In this section, we are defining the definite $ pq $-integral, by using formula \eqref{pqintdef}.
As in proof of the Theorem \ref{coninttheo}, we examine the following three cases, to gain our aim:

\textbf{Case 1.} Let $ 1<a<b $ and function $ f $ is defined on interval $ (1,b] $. 

\begin{definition}\label{defintcase1}
The definite $ pq $-integral of $ f(x) $ on interval $ (1,b] $ is given by
\begin{equation}\label{defintcase1eqn1}
\int\limits_{1}^b f(x) d_{pq}x=\lim\limits_{N \to \infty} 	\sum_{j=0}^{N} \left( b^{q^{j}p^{-j}}-b^{q^{j+1}p^{-j-1}}\right)f\left(b^{q^{j}p^{-j-1}}\right)
\end{equation}
and

\begin{equation}\label{defintcase1ii}
\int\limits_{a}^b f(x) d_{pq}x=\int\limits_{1}^b f(x) d_{pq}x-\int\limits_{1}^a f(x) d_{pq}x.
\end{equation}
\end{definition}

\begin{example}
Let $ b=4 $ and $ f(x)=c $ where $ c $ is constant.
\begin{eqnarray*}
\int\limits_{1}^4 c d_{pq}x&=&\lim\limits_{N \to \infty} 	\sum_{j=0}^{N} \left( 4^{q^{j}p^{-j}}-4^{q^{j+1}p^{-j-1}}\right)f\left(4^{q^{j}p^{-j-1}}\right) \\
&=& c \lim\limits_{N \to \infty} \left[ \left(4- 4^{q p^{-1}}\right)+\left(4^{q p^{-1}}- 4^{q^{2} p^{-2}}\right)+\ldots+\left(4^{q^{N} p^{-N}}- 4^{q^{N+1} p^{-N-1}}\right) \right]\\
&=& c \lim\limits_{N \to \infty} \left[ 4- 4^{q^{N+1} p^{-N-1}} \right]=c(4-1)=3c,
\end{eqnarray*}
and if $ a=3 $,

\begin{equation*}
\int\limits_{3}^4 c d_{pq}x=\int\limits_{1}^4 c d_{pq}x-\int\limits_{1}^3 c d_{pq}x=3c-2c=c.
\end{equation*}

\end{example}
\begin{example}
Let $ b=3 $ and $\displaystyle{ f(x)=\frac{ln(x)}{x^p-x^q} }$.
\begin{equation*}
\int\limits_{1}^3 f(x) d_{pq}x=\sum_{j=0}^{\infty} \left( 3^{q^j p^{-j}}- 3^{q^{j+1} p^{-j-1}} \right)\frac{\ln(3^{q^{j}p^{-j}})}{ 3^{q^j p^{-j}}- 3^{q^{j+1} p^{-j-1}}}=\sum_{j=0}^{\infty} q^{j}p^{-j}\ln 3=\frac{p \ln 3 }{p-q}
\end{equation*}
\end{example}

\textbf{Case 2.} Let $ 0<a<b<1 $ and function $ f $ is defined on $ [b,1) $.

\begin{definition}\label{defintcase2}
	The definite $ pq $-integral of $ f(x) $ on interval  $ [b,1) $ is given by
	\begin{eqnarray}\label{defintcase2eqn1}
	\int\limits_{b}^1 f(x) d_{pq}x&=& \lim\limits_{N \to \infty} \sum_{j=0}^{N} \left( b^{q^{j+1}p^{-j-1}}-b^{q^{j}p^{-j}}\right)f\left(b^{q^{j}p^{-j-1}}\right)\notag \\
	&=&	\sum_{j=0}^{\infty} \left( b^{q^{j+1}p^{-j-1}}-b^{q^{j}p^{-j}}\right)f\left(b^{q^{j}p^{-j-1}}\right)
	\end{eqnarray}

\end{definition}

\begin{example}
	Let $ b=\frac{1}{3} $ and $ f(x)=c $ where $ c $ is constant.
	\begin{eqnarray*}
		\int\limits_{\frac{1}{3}}^1 c d_{pq}x&=&\lim\limits_{N \to \infty} 	\sum_{j=0}^{N} \left( \left(\frac{1}{3}\right)^{q^{j+1}p^{-j-1}}-\left(\frac{1}{3}\right)^{q^{j}p^{-j}}\right)c \\
		&=& c \lim\limits_{N \to \infty} \left[ \left(\left(\frac{1}{3}\right)^{q p^{-1}}- \frac{1}{3}\right)+\left(\left(\frac{1}{3}\right)^{q^{2} p^{-2}}- \left( \frac{1}{3}\right)^{q^{} p^{-1}}\right)+ \ldots \right. \\
		&& \qquad\qquad\qquad\qquad\qquad \left. \ldots+\left(\left(\frac{1}{3}\right)^{q^{N+1} p^{-N-1}}-\left(\frac{1}{3}\right)^{q^{N} p^{-N}} \right) \right]\\
		&=& c \lim\limits_{N \to \infty} \left[ -\left(\frac{1}{3}\right)+ \left(\frac{1}{3}\right)^{q^{N+1} p^{-N-1}} \right]=c\left(-\frac{1}{3}+1\right))=\frac{2c}{3}.
	\end{eqnarray*}

\end{example}
\begin{remark}
	The definite $ pq $-integrals \eqref{defintcase1} and \eqref{defintcase2} are also denoted by
	\begin{eqnarray}
		\int\limits_{1}^b f(x) d_{pq}x&=& \operatorname{I_{pq}^{+}}f(b), \\
		\int\limits_{b}^1 f(x) d_{pq}x&=& \operatornamewithlimits{I_{pq}^{--}}f(b).
	\end{eqnarray}
\end{remark}

\textbf{Case 3.}  Let $ 0<a<b<1$ and function $ f $ is defined on $ (0,b] $.

\begin{definition}\label{defintcase3}
	The definite $ pq $-integral of $ f(x) $ on interval  $(0,b] $ is given by
	\begin{eqnarray}\label{defintcase3eqn1}
\operatorname{I_{pq}^{ }} f(b)&=&\int\limits_{0}^b f(x) d_{pq}x \notag \\
	&=&\lim\limits_{N \to \infty} \sum_{j=0}^{N} \left( b^{q^{-j}p^{j}}-b^{q^{-j-1}p^{j+1}}\right)f\left(b^{q^{-j}p^{j+1}}\right)\notag \\ 
	&=&	\sum_{j=0}^{\infty} \left( b^{q^{-j}p^{j}}-b^{q^{-j-1}p^{j+1}}\right) f\left(b^{q^{-j}p^{j+1}}\right)
	\end{eqnarray}
	and
	\begin{equation}\label{defintcase3ii}
	\int\limits_{a}^b f(x) d_{pq}x=\int\limits_{0}^b f(x) d_{pq}x-\int\limits_{0}^a f(x) d_{pq}x.
	\end{equation}
\end{definition}

\begin{example}
Let $ a=\frac{1}{4}, b=\frac{1}{2} $ and $ f(x)=c $ where $ c $ is constant.
\begin{eqnarray*}
	\int\limits_{0}^{\frac{1}{2}} c d_{pq}x&=&\lim\limits_{N \to \infty} 	\sum_{j=0}^{N} \left( \left(\frac{1}{2}\right)^{q^{-j}p^{j+1}}-\left(\frac{1}{2}\right)^{q^{-j}p^{j+1}}\right)c \\
	&=& c \lim\limits_{N \to \infty} \left[ \left(\frac{1}{2}-\left(\frac{1}{2}\right)^{q^{-1} p^{1}}\right)+\left(\left(\frac{1}{2}\right)^{q^{-1} p^{1}}- \left( \frac{1}{2}\right)^{q^{-2} p^{2}}\right)+ \right. \\
	&& \qquad\qquad\quad \quad \qquad\qquad \left. \ldots+\left(\left(\frac{1}{2}\right)^{q^{-N} p^{N}}-\left(\frac{1}{2}\right)^{q^{-N-1} p^{N+1}} \right) \right]\\
	&=& c \lim\limits_{N \to \infty} \left[ \left(\frac{1}{2}\right)+ \left(\frac{1}{2}\right)^{q^{-N-1} p^{N+1}} \right]=\frac{c}{2}.
\end{eqnarray*}
As a similar way,
\[
\int\limits_{0}^{\frac{1}{4}} c d_{pq}x=\frac{c}{4},
\]
thus we have
\[
\int\limits_{\frac{1}{4}}^{\frac{1}{2}} c d_{pq}x=\int\limits_{0}^{\frac{1}{2}} c d_{pq}x-\int\limits_{0}^{\frac{1}{4}} c d_{pq}x=\frac{c}{4}
\]
\end{example}

\begin{definition}\label{defintcase4}
	Let  $ 0 \leq a<1<b$ be. The definite $ pq $-integral of $ f(x)$ is given by
	\begin{eqnarray}\label{defintcase4eqn1}
\int\limits_{a}^{b} f(x) d_{pq}x= \int\limits_{a}^1 f(x) d_{pq}x+\int\limits_{1}^b f(x) d_{pq}x.
	\end{eqnarray}
\end{definition}

\begin{corollary}
 Definitions of $ pq $-integrals and \eqref{pqintgeneral1} imply a more general formula:
\begin{itemize}
	\item[i.)] If $ b>1 $ then
	\[
	\int\limits_{1}^b f(x) d_{pq}g(x)= \sum_{j=0}^{\infty} f\left(b^{q^j p^{-j-1}}\right) \left( g\left(b^{q^j p^{-j}}\right)-g\left(b^{q^{j+1} p^{-j-1}}\right)\right).
	\]
		\item[ii.)] If $0<b<1 $ then
		
	\[
	\int\limits_{0}^b f(x) d_{pq}g(x)= \sum_{j=0}^{\infty} f\left(b^{q^{-j} p^{j+1}}\right) \left( g\left(b^{q^{-j} p^{j}}\right)-g\left(b^{q^{-j-1} p^{j+1}}\right)\right).
	\]	
	
\end{itemize}

\end{corollary}

\begin{definition}
The $ pq $-integral of higher order of function $ f $ is given by
\[
\left( \operatorname{I_{pq}^0} f\right)(x)=f(x), \quad \left( \operatorname{I_{pq}^n} f\right)(x)=\operatorname{I_{pq}}\left(\operatorname{I_{pq}^{n-1}} f\right)(x), \quad n\in \mathbb{N}.
\]
\end{definition}

\section{ Improper $ pq $-integral}

In this section,  we need to take a look at the improper $ pq $-integral of $ f(x) $ and  sufficient condition for its convergence.

Let $ 0<q<p<1 $, thus $\left(\frac{q}{p}\right)^{-1}>1$  and set  $b= \left(\frac{q}{p}\right)^{-1}$. For any $ j \in {0, \pm 1, \pm 2, \ldots} $, we have $b^{\frac{q^{j}}{p^{j}}}>1,\quad b^{\frac{q^{j+1}}{p^{j+1}}}<b^{\frac{q^{j}}{p^{j}}} $ and thus according \eqref{defintcase1ii}, we first examine integral

\begin{eqnarray*}
\int\limits_{b^{q^{j+1}p^{-j-1}}}^{b^{q^{j}p^{-j}}} f(x) d_{pq}x&=&\int\limits_{1}^{b^{q^{j}p^{-j}}} f(x) d_{pq}x-\int\limits_{1}^{b^{q^{j+1}p^{-j-1}}} f(x) d_{pq}x \\
&=& \sum_{k=0}^{\infty} \left( \left( b^{q^{j}p^{-j}}\right)^{q^{k}p^{-k}}-\left( b^{q^{j}p^{-j}}\right)^{q^{k+1}p^{-k-1}}\right)f\left(\left( b^{q^{j}p^{-j}}\right)^{q^{k}p^{-k-1}}\right) \\
&&-\sum_{k=0}^{\infty} \left( \left( b^{q^{j+1}p^{-j-1}} \right)^{q^{k}p^{-k}}-\left( b^{q^{j+1}p^{-j-1}}\right)^{q^{k+1}p^{-k-1}}\right)f\left(\left( b^{q^{j+1}p^{-j-1}}\right)^{q^{k}p^{-k-1}}\right) \\
&=&\sum_{k=0}^{\infty} \left(  b^{q^{j+k}p^{-j-k}}- b^{q^{j+k+1}p^{-j-k-1}}\right)f\left( b^{q^{j+k}p^{-j-k-1}}\right) \\
&&-\sum_{k=0}^{\infty} \left(  b^{q^{j+k+1}p^{-j-k-1}}- b^{q^{j+k+2}p^{-j-k-2}}\right)f\left( b^{q^{j+k+1}p^{-j-k-2}}\right),
\end{eqnarray*}
and thus, implies that
\[
\int\limits_{b^{q^{j+1}p^{-j-1}}}^{b^{q^{j}p^{-j}}} f(x) d_{pq}x=\left(  b^{q^{j}p^{-j}}- b^{q^{j+1}p^{-j-1}}\right)f\left( b^{q^{j}p^{-j-1}}\right).
\]
\begin{definition}\label{impdef1}
Let $ 0<q<p<1 $ and set $b=\left(\frac{q}{p}\right)^{-1}$. The improper $ pq $-integral of $ f(x) $ on $[1, \infty)  $ is given by

\begin{eqnarray}
\int\limits_{1}^{\infty} f(x) d_{pq}x&=&\sum\limits_{j=- \infty}^{\infty} \int\limits_{b^{q^{j+1}p^{-j-1}}}^{b^{q^{j}p^{-j}}} f(x) d_{pq}x \notag \\
&=& \sum\limits_{j=- \infty}^{\infty}\left(  b^{q^{j}p^{-j}}- b^{q^{j+1}p^{-j-1}}\right)f\left( b^{q^{j}p^{-j-1}}\right) \notag \\
&=&\sum\limits_{j=0}^{\infty}\left(  b^{q^{j}p^{-j}}- b^{q^{j+1}p^{-j-1}}\right)f\left( b^{q^{j}p^{-j-1}}\right) \notag \\
&&\quad + \sum\limits_{j=1}^{\infty}\left(  b^{q^{-j}p^{j}}- b^{q^{-j-1}p^{j+1}}\right)f\left( b^{q^{-j}p^{j+1}}\right). 
\end{eqnarray}
\end{definition}

\begin{definition}
If $ 0<q<p<1 $, then for any $ j \in {0, \pm 1, \pm 2, \ldots} $, we observe $ p^{q^{j}p^{-j}} \in (0,1), \quad  p^{q^{j}p^{-j}}< p^{q^{j+1}p^{-j-1}}$ and

\[
\int\limits_{0}^{1} f(x) d_{pq}x=\sum\limits_{j=- \infty}^{\infty} \left( p^{q^{j+1}p^{-j-1}}- p^{q^{j}p^{-j}} \right) f\left( p^{q^{-j}p^{j+1}}\right).
\]
Because, on account of \eqref{defintcase3ii}

\begin{eqnarray*}
\int\limits_{p^{q^{j}p^{-j}}}^{p^{q^{j+1}p^{-j-1}}} f(x) d_{pq}x&=&\int\limits_{0}^{p^{q^{j+1}p^{-j-1}}} f(x) d_{pq}x~-\int\limits_{0}^{p^{q^{j}p^{-j}}} f(x) d_{pq}x\\
&=& \left( p^{q^{j+1}p^{-j-1}}- p^{q^{j}p^{-j}} \right) f\left( p^{q^{-j}p^{j+1}}\right)
\end{eqnarray*}
It follows that,
\begin{eqnarray*}
\int\limits_{0}^{1} f(x) d_{pq}x&=&\sum\limits_{j=- \infty}^{\infty}\int\limits_{p^{q^{j}p^{-j}}}^{p^{q^{j+1}p^{-j-1}}} f(x) d_{pq}x\\
&=&\sum\limits_{j=- \infty}^{\infty} \left( p^{q^{j+1}p^{-j-1}}- p^{q^{j}p^{-j}} \right) f\left( p^{q^{-j}p^{j+1}}\right).
\end{eqnarray*}

\end{definition}

\begin{definition}
If $ 0<q<p<1 $ then the improper $ pq $-integral of $ f(x) $ on $ [0, \infty) $ is given by
\[
\int\limits_{0}^{\infty} f(x) d_{pq}x=\int\limits_{0}^{1} f(x) d_{pq}x+\int\limits_{1}^{\infty} f(x) d_{pq}x.
\]
\end{definition}

\begin{definition}\label{impdefinf}
If $ 0<q<p<1 $ then the improper $ pq $-integral of $ f(x) $ on $ [a, \infty) $ is given as  follows:
	\begin{itemize}
		\item[i.)] If $ 0<a<1 $ then
	    \[
		\int\limits_{a}^{\infty} f(x) d_{pq}x=\int\limits_{a}^{1} f(x) d_{pq}x+\int\limits_{1}^{\infty} f(x) d_{pq}x.
		\]
		\item[ii.)]  If $ a>1 $ then
	\end{itemize}
	\begin{equation*}
	\int\limits_{a}^{\infty} f(x) d_{pq}x=\lim\limits_{N \to \infty} \sum\limits_{j=1}^{N} \int\limits_{a^{q^{-j+1}p^{^j-1}}}^{a^{q^{-j}p^{j}}} f(x) d_{pq}x.
	\end{equation*}
\end{definition}

Now, we wish to investigate about convergence of the improper $ pq $-integral.

\begin{proposition}
The improper  $ pq $-integral of $ f(x) $ defined above converges on  $ [1, \infty) $ if $ f $ satisfies that  for $0<r<\infty$
\[
|f(x)|< \min\left\{rx^{^\alpha}, |x^{p^{-1}}-x^{q^{-1}}|^{-1}( \ln x)^{2\alpha} \right\}, 
\]
in neighborhood of $ x=1 $ with some $ 0 \leq \alpha <1 $ and for sufficiently large $ x $ with some $ -\varepsilon \leq \alpha <0 $ where $ \varepsilon $ is a small positive number.
	
\end{proposition}

\begin{proof}
Consider $ b= \left(\frac{q}{p}\right)^{-1}$. On account of Definition \ref{impdef1}, the proof falls naturally into two parts:
\begin{eqnarray*}
\int\limits_{1}^{\infty} f(x) d_{pq}x&=&\sum\limits_{j=0}^{\infty}\left(  b^{q^{j}p^{-j}}- b^{q^{j+1}p^{-j-1}}\right)f\left( b^{q^{j}p^{-j-1}}\right) \\
&&+ \sum\limits_{j=1}^{\infty}\left(  b^{q^{-j}p^{j}}- b^{q^{-j-1}p^{j+1}}\right)f\left( b^{q^{-j}p^{j+1}}\right).
\end{eqnarray*}

Under  the conditions stated above  and also by  Theorem \ref{coninttheo}, the convergence of the first sum is proved. For the second sum, assume that for large $ x $, we have $ |f(x)|<  |x^{p^{-1}}-x^{q^{-1}}|^{-1}( \ln x)^{2\alpha}  $ where $ -\varepsilon \leq \alpha <0 $. Then, for sufficiently large $ j $ we get

\[
\left| f\left( b^{q^{-j}p^{j+1}}\right) \right|< \left(  b^{q^{-j}p^{j}}- b^{q^{-j-1}p^{j+1}}\right)^{-1} \left( \ln b^{q^{-j}p^{j+1}} \right)^{2 \alpha}.
\]
Hence
\begin{eqnarray*}
\left| \left(  b^{q^{-j}p^{j}}- b^{q^{-j-1}p^{j+1}}\right)  f\left( b^{q^{-j}p^{j+1}}\right) \right| &\leq& \left(  b^{q^{-j}p^{j}}- b^{q^{-j-1}p^{j+1}}\right) \left(  b^{q^{-j}p^{j}}- b^{q^{-j-1}p^{j+1}}\right)^{-1} \left( \ln b^{q^{-j}p^{j+1}} \right)^{2 \alpha}
\\
&=& \left( \ln b^{q^{-j}p^{j+1}} \right)^{2 \alpha} = \left( \ln  b \right)^{2 \alpha} \left( p^{2\alpha} \right)^{j+1} \left( q^{-2\alpha} \right)^{j}
\end{eqnarray*}
Therefore, the second sum is also majorized by a convergent geometric series, and thus  converges.
\end{proof}

\section{Fundamental Theorem of $ pq $-Calculus}
In this section, we give the fundamental theorem of $ pq $-calculus which explains relation between $ pq $-derivative and $ pq $-integral. This relation will be proved once we prove some lemmas below.

\begin{lemma}\label{funlemma1}
Assume that $ x>1 $ . Then $ \opD \operatorname{I_{pq}^{+}}f(x)= f(x)$, and also $ \operatorname{I_{pq}^{+}} \opD f(x)=f(x)-f(1) $ holds  if  function $ f $ is continuous at $ x=1 $.
\end{lemma}

\begin{proof}
Since by  definitions of $ pq $-derivative and $ pq $-integral, we get
\begin{equation*}
\operatorname{I_{pq}^{+}} f(x)=\int\limits_{1}^x f(s)d_{pq}s=	\sum_{j=0}^{\infty} \left( x^{q^{j}p^{-j}}-x^{q^{j+1}p^{-j-1}}\right)f\left(x^{q^{j}p^{-j-1}}\right).
\end{equation*}

Hence,
\begin{eqnarray*}
\opD^{} \operatorname{I_{pq}^{+}} f(x)&=& \frac{\operatorname{I_{pq}^{+}}f(x^{p})-\operatorname{I_{pq}^{+}} f(x^{q})	}{x^{p}-x^{q}} \\
&=& \frac{\sum\limits_{j=0}^{\infty} \left( x^{q^{j}p^{-j+1}}-x^{q^{j+1}p^{-j}}\right)f\left(x^{q^{j}p^{-j}}\right)-\sum\limits_{j=0}^{\infty} \left( x^{q^{j+1}p^{-j}}-x^{q^{j+2}p^{-j-1}}\right)f\left(x^{q^{j+1}p^{-j-1}}\right)}{x^{p}-x^{q}} \\
&=& \frac{\left[ (x^p-x^q)f(x)+(x^q-x^{q^{2}p^{-1}})f(x^{qp^{-1}})+(x^{q^{2}p^{-1}}-x^{q^{3}p^{-2}})f(x^{q^{2}p^{-2}})+\ldots  \right] }{x^{p}-x^{q}} \\ && -\frac{\left[  (x^q-x^{q^{2}p^{-1}})f(x^{qp^{-1}})+(x^{q^{2}p^{-1}}-x^{q^{3}p^{-2}})f(x^{q^{2}p^{-2}})+\ldots   \right]}{x^{p}-x^{q}} \\
&=&\frac{(x^p-x^q)f(x)}{x^{p}-x^{q}}=f(x).
\end{eqnarray*}

Also,
\begin{eqnarray*}
 \operatorname{I_{pq}^{+}} \opD f(x)&=& \lim\limits_{N \to \infty} 	\sum_{j=0}^{N} \left( x^{q^{j}p^{-j}}-x^{q^{j+1}p^{-j-1}}\right)\opD f\left(x^{q^{j}p^{-j-1}}\right) \\
 &=&\lim\limits_{N \to \infty} 	\sum_{j=0}^{N} \left( x^{q^{j}p^{-j}}-x^{q^{j+1}p^{-j-1}}\right) \frac{f(x^{q^{j}p^{-j}})-f(x^{q^{j+1}p^{-j-1}})}{x^{q^{j}p^{-j}}-x^{q^{j+1}p^{-j-1}}} \\
  &=&\lim\limits_{N \to \infty} 	\sum_{j=0}^{N} \left[ f(x^{q^{j}p^{-j}})-f(x^{q^{j+1}p^{-j-1}}) \right] \\
  &=&\lim\limits_{N \to \infty} \left[ f(x)-f(x^{q^{}p^{-1}}) + f(x^{q^{}p^{-1}})-f(x^{q^{2}p^{-2}})+ \ldots + f(x^{q^{N}p^{-N}})-f(x^{q^{N+1}p^{-N-1}}) \right]  \\
  &=&\lim\limits_{N \to \infty} \left[ f(x)-f(x^{q^{N+1}p^{-N-1}}) \right] \\
  &=&f(x)-f(1).
\end{eqnarray*}

We noted that the last equality holds because of continuity of $ f $ at $ x=1 $

\end{proof}
By a similar argument,  the following lemmas can be easily obtained.

\begin{lemma}\label{funlemma2}
Assume that $ x \in (0,1)$. In this case  $ \opD^{} \operatorname{I_{pq}^{--}}f(x)= -f(x)$, and also  $ \operatorname{I_{pq}^{--}} \opD f(x)=f(1)-f(x)  $ holds  if function $ f $ is continuous at $ x=1 $. 
\end{lemma}

\begin{lemma}\label{funlemma3}
Assume that $ x \in (0,1)$ . In this case $ \operatorname{I_{pq}^{ }} $ is defined by
	\[
	\operatorname{I_{pq}^{ }}f(x)=\int\limits_{0}^x f(s)d_{pq}s,
	\]
	then  $ \opD^{} \operatorname{I_{pq}^{ }}f(x)= f(x)$, and also $ \operatorname{I_{pq}^{ }} \opD f(x)=f(x)-f(0).  $ holds  if function $ f $ is continuous at $ x=0 $.
\end{lemma}

The following theorem is a reformulation of  the fundamental theorem of ordinary calculus, for $ pq $-calculus.

\begin{theorem}\label{funtheo}
 If  function $ F(x) $ is an antiderivative of $ f(x) $ and $ F(x) $ is continuous at $ x=0 $ and $ x=1 $, then for every $ 0 \leq a< b \leq \infty $, one can obtain 
\begin{equation}\label{funeqn1}
\int\limits_{a}^b f(x)d_{pq}x=F(b)-F(a).	
\end{equation}
\end{theorem}

\begin{proof}
The proof falls naturally into three cases:

\textbf{Case 1.} Assume that $ a, b $ are finite and $ 1<a<b $. By assumptions, we have $  \opD F(x)=f(x) $. By Lemma \ref{funlemma1}, we obtain
\begin{equation*}
	F(x)-F(1)=\operatorname{I_{pq}^{+}}f(x)=\int\limits_{1}^x f(s)d_{pq}s,
\end{equation*}
which yields
\begin{equation*}
\int\limits_{1}^a f(s)d_{pq}s=	F(a)-F(1) ~\text{and}~ \int\limits_{1}^b f(s)d_{pq}s=	F(b)-F(1).
\end{equation*}

By \eqref{defintcase1ii}, thus we deduce that

\begin{equation*}
\int\limits_{a}^b f(x)d_{pq}x=F(b)-F(a).
\end{equation*}

\textbf{Case 2.} Assume that $ 0 <a<b<1 $. By assumptions and Lemma \ref{funlemma3}, we have
\begin{equation*}
F(x)-F(0)=\operatorname{I_{pq}^{ }}f(x)=\int\limits_{0}^x f(s)d_{pq}s,
\end{equation*}
which yields
\begin{equation*}
\int\limits_{0}^a f(s)d_{pq}s=	F(a)-F(0), \quad \int\limits_{0}^b f(s)d_{pq}s=	F(b)-F(0).
\end{equation*}

By  \eqref{defintcase3ii}, thus we deduce that

\begin{equation*}
\int\limits_{a}^b f(x)d_{pq}x=F(b)-F(a).
\end{equation*}

\textbf{Case 3.} Assume that $ b $ is finite and $ 0<a<1<b $. On account of Definition \ref{defintcase4} and also by Lemma \ref{funlemma2}, it can be deduced that

\begin{equation*}
\int\limits_{a}^1 f(x)d_{pq}x =\operatorname{I_{pq}^{--}}f(a)=\operatorname{I_{pq}^{--}} \opD f(a)=F(1)-F(a).
\end{equation*}

Similarly,
\begin{equation*}
\int\limits_{1}^b f(x)d_{pq}x =\operatorname{I_{pq}^{+}}f(b)=\operatorname{I_{pq}^{+}}\opD f(b)=F(b)-F(1).
\end{equation*}

Thus, we conclude finally that 

\begin{equation*}
\int\limits_{a}^b f(x)d_{pq}x=F(b)-F(a).
\end{equation*}
\end{proof}

\begin{remark}
	Assuming that $ a>1 $ and setting $ b=+\infty $, by the Definition \ref{impdefinf}, we have
\begin{eqnarray*}
\int\limits_{a}^{\infty} f(x)d_{pq}x &=& \lim\limits_{N \to \infty} \sum\limits_{j=1}^{N} \int\limits_{a^{q^{-j+1}p^{j-1}}}^{a^{q^{-j}p^{j}}} f(x) d_{pq}x \\
&=& \lim\limits_{N \to \infty} \sum\limits_{j=1}^{N} \left( F(a^{q^{-j}p^{j}})-F(a^{q^{-j+1}p^{j-1}})   \right) \\
&=& \lim\limits_{N \to \infty} \left( F(a^{q^{-N}p^{N}})-F(a)   \right),
\end{eqnarray*}
and if $ \lim\limits_{x \to \infty} F(x) $ exists then for $ b=+\infty $  \eqref{funeqn1} holds, too.
\end{remark}

\begin{corollary}\label{funcor1}
If $ f(x) $ is continuous at $ x=0 $ and $ x=1 $, then we obtain
\begin{equation*}
\int\limits_{a}^b \opD f(x)d_{pq}x=f(b)-f(a).
\end{equation*}
\end{corollary}

\begin{corollary}\label{funcor1}
	If $ f(x) $ and $ g(x) $ are continuous at $ x=0 $ and $ x=1 $, then we deduce the following formulas, which are the formulas of $ pq $-integration by parts.
	\begin{eqnarray}
	\int\limits_{a}^b f(x^{q})d_{pq}g(x)&=&f(b)g(b)-f(a)g(a)-\int\limits_{a}^b g(x^{p})d_{pq}f(x), \label{bypart1}\\
	\int\limits_{a}^b f(x^{p})d_{pq}g(x)&=&f(b)g(b)-f(a)g(a)-\int\limits_{a}^b g(x^{q})d_{pq}f(x) \label{bypart2}
	\end{eqnarray}
\end{corollary}

\begin{proof}
	By using the product rule \eqref{prrule1}, we see that
\begin{equation*}
\int\limits_{a}^b \opD(fg)(x)d_{pq}x=\int\limits_{a}^b \left( g(x^p)\opD f(x)+f(x^q)\opD g(x) \right) d_{pq}x
\end{equation*}	

By previous corollary, it yields
\begin{equation*}
	f(b)g(b)-f(a)g(a)= \int\limits_{a}^b f(x^q)\opD g(x) d_{pq}x + \int\limits_{a}^b g(x^p)\opD f(x) d_{pq}x.
\end{equation*}	

Since $ \opD g(x) d_{pq}x= d_{pq}g(x) $, it gives \eqref{bypart1}. Similarly, by using the product rule \eqref{prrule2}, one can get \eqref{bypart2}.

\end{proof}

\end{document}